\documentclass[]{amsart}
\usepackage{amssymb,amsmath,amscd}
\usepackage{epsfig,graphicx}

\usepackage{mathrsfs}
\usepackage{amsthm}
\usepackage{latexsym}
\usepackage{enumitem}

\usepackage{stmaryrd}

\usepackage{palatino}
\usepackage[T1]{fontenc}  

\usepackage[margin=0.8in]{geometry}

\theoremstyle{plain}
\newtheorem{theorem}[equation]{Theorem}
\newtheorem{corollary}[equation]{Corollary}

\newcommand{\real}{{\mathbb R}}

\newcommand{\nats}{{\mathbb N}}

\newcommand{\zed}{{\mathbb Z}}

\newcommand{\CF}{{\sf CF}}
\newcommand{\CharFunc}{{\mathbf{1}}}

\renewcommand{\dim}{{\sf dim}}              
\newcommand{\Disc}{{\mathbb D}}   

\newcommand{\Field}{{\mathbb F}}        

\newcommand{\Grading}{\bullet}       

\newcommand{\Radon}{{\mathcal R}}       

\newcommand{\sphere}{{\mathbb S}}   
\newcommand{\Sphere}{{\mathbb S}}   

\title{Persistent homology and Euler integral transforms}
\author{Robert Ghrist}
\address{Departments of Mathematics and Electrical \& Systems Engineering, University of Pennsylvania}
\email{ghrist@math.upenn.edu}
\thanks{Work supported by the Office of the Assistant Secretary of Defense Research \& Engineering through ONR N00014-16-1-2010.}
\author{Rachel Levanger}
\address{Department of Electrical \& Systems Engineering, University of Pennsylvania}
\email{levanger@seas.upenn.edu}
\author{Huy Mai}
\address{Department of Mathematics, University of Pennsylvania}
\email{huymai@sas.upenn.edu}

\begin{document}
\maketitle
\begin{abstract}
The Euler calculus -- an integral calculus based on Euler characteristic as a valuation on constructible functions -- is shown to be an incisive tool for answering questions about injectivity and invertibility of recent transforms based on persistent homology for shape characterization.
\end{abstract}

\section{Injective transforms based on persistent homology.}
The past fifteen years have witnessed the rise of Topological Data Analysis as a novel means of extracting structure from data. In its most common form, data means a point cloud sampled from a subset of Euclidean space, and structure comes from converting this to a filtered simplicial complex and applying persistent homology (see \cite{C,EH} for definitions and examples). This has proved effective in a number of application domains, including genetics, neuroscience, materials science, and more.

Recent work considers an inverse problem for shape reconstruction based on topological data. In particular, \cite{TMB} defines a type of transform which is based on persistent homology as follows. Given a (reasonably tame) subspace $X\subset\real^n$, one considers a function from $\sphere^{n-1}\times\nats$ to the space of persistence modules over a field $\Field$. For those familiar with the literature, this {\em persistent homology transform} records sublevelset homology barcodes in all directions ($\Sphere^{n-1}$) and all gradings ($\nats$).
The paper \cite{TMB} contains the following contributions.
\begin{enumerate}
\item For compact nondegenerate shapes in $\real^2$ and compact triangulated surfaces in $\real^3$, the persistent homology transform is injective; thus one can in principle reconstruct the shapes based on the image in the space of persistence modules. The proof is an algorithm.
\item It is claimed that the proof survives reduction to the Euler characteristic, so that knowing all Euler characteristics of the intersection of the shape with all half-spaces in $\real^2$ or $\real^3$ (resp.) yields a likewise injective transform.
\item Certain results on {\em sufficient statistics} follow from this injectivity, which are then applied to shape characterization (see also \cite{CMCMR}). This is effected by discretizing the Euler characteristic transform both in direction and along the filtration.
\end{enumerate}

This note reformulates the persistent homology transform of \cite{TMB} in terms of Euler calculus on constructible functions. Though a more abstract framework, the theory effortlessly permits the following results.
\begin{enumerate}
\item The Euler characteristic reduction of the persistent homology transform extends to an integral transform on constructible functions.
\item This integral transform has an explicit inverse, with no restrictions on dimension, manifold structure, or nondegeneracy (beyond constructibility).
\item This integral transform is shown to be but one of several invertible transforms that characterizes shapes with topological data.
\end{enumerate}

{\em Nota bene:} upon posting this preprint, we learned of independent work-in-progress. The work of Curry {\em et al.} \cite{CMT} contains results (1) and (2) above with a similar proof. In addition, \cite{CMT} has a number of further interesting results concerning reconstruction from finite samplings. For a restricted class of graphs, work of Belton {\em et al.} also has a finite reconstruction algorithm \cite{BF+}. 

\section{Euler calculus.}
Euler characteristic is an integer-valued ``compression'' of a finitely-nonzero sequence $V_\Grading$ of finite-dimensional vector spaces over a field $\Field$ given by the alternating sum of dimensions. Among complexes, Euler characteristic is an invariant of quasi-isormorphism, meaning that for $C_\Grading$ a complex of vector spaces and $H_\Grading$ its homology, $\chi(H_\Grading)=\chi(C_\Grading)$. On compact cell complexes, $\chi$ is well-defined and a homotopy invariant. Euler characteristic is additive on compact cell complexes, meaning that for $A$ and $B$ such, $\chi(A\cup B) = \chi(A) + \chi(B) - \chi(A\cap B)$.

It is profitable to pass from the realm of compact cell complexes to more general {\em definable} or {\em constructible} subsets of $\real^n$  by using compactly-supported cohomology. This, combined with results from {\em o-minimal structures} \cite{vdd} makes it trivial to work with an additive and homeomorphism-invariant Euler characteristic on definable sets. For the reader unfamiliar with the o-minimal theory, it suffices to substitute {\em semialgebraic} for {\em definable} or {\em constructible} in what follows.

For $X$ a definable subset of Euclidean space, the {\em constructible functions} on $X$ are functions $h\colon X\to\zed$ that have definable (and locally finite) level sets. The set of constructible functions, $\CF(X)$, has the structure of a sheaf with the obvious restriction maps.\footnote{This structure, though very helpful for generating clean definitions, can be ignored by the reader for whom sheaves are unfamiliar.} The Euler integral on $X$ is simply the functional
\begin{equation}
\int_X\cdot \,d\chi\colon \CF(X)\to\zed
\quad {\textrm{taking}} \quad
\CharFunc_{\sigma}\mapsto(-1)^{\dim\ \sigma}
\end{equation}
for each (open) definable simplex $\sigma$. As all definable sets are finitely definably triangulated, the Euler integral is well-defined and additive. Euler calculus possesses a Fubini Theorem, a convolution operation, and much more.  For a thorough introduction, see \cite{CGR}.

\section{Euler-Radon transform \& inversion.}
%
%
%
%

The first application of Euler calculus to integral transforms was given by Schapira in a seminal paper \cite{Schapira} that defined a topological Radon transform and gave conditions for an inverse to exist. Our summary uses compactly-supported constructible functions $\CF_c(-)$ and follows the reformulation in \cite{BGL} to weighted kernels. Consider a pair $(X, Y)$ of definable spaces and $K\in\CF(X\times Y)$ a kernel --- a constructible function on the product. The Radon transform $\Radon_K\colon\CF(X)\to\CF(Y)$ is defined explicitly via the formula
\begin{equation}
\label{eq:Radon}
    (\Radon_K h)(y) = \int_X h(x)\,K(x,y)\,d\chi(x) .
\end{equation}
The principal result of \cite{Schapira} is the following. Consider a second kernel $K'\in\CF(Y\times X)$ with Radon transform $\Radon_{K'}\colon\CF(Y)\to\CF(X)$. If there are constants $\lambda, \mu$ such that
\begin{equation}
\label{eq:RadonHypothesis}
    \int_Y K(x,y)K'(y,x') \, d\chi(y) = (\mu-\lambda)\delta_{\Delta} + \lambda ,
\end{equation}
for $\Delta\subset X\times X$ the diagonal, then
\begin{equation}
\label{eq:RadonInverse}
    (\Radon_{K'}\circ\Radon_K)h = (\mu - \lambda)h + \lambda\left( \int_X h\, d\chi \right)\CharFunc_{X} .
\end{equation}
Thus, when $\lambda\neq\mu$, one can recover $h$ exactly from the inverse transform (followed by the appropriate rescaling).

The point of this note is to show that working with Euler integral transforms is preferable to mapping a set into a space of persistence modules, as the Euler transform provides a more efficient representation that yields full invertibility, not merely injectivity.

%
%
%

\section{Inversion for the sublevelset Euler integral transform.}
The persistent homology transform of \cite{TMB} is easily converted into a Radon integral transform. Let $X=\real^n$ and $Y=\Sphere^{n-1}\times\real$ with kernel $K$ the indicator function on the set $\{(x,(\xi, t)) \colon x\cdot\xi\leq t\}$. Given the resemblance to sublevelset filtrations in persistent homology, we denote this the {\em sublevelset Euler integral transform}.

\begin{theorem}
The sublevelset Euler integral transform $\Radon_K\colon\CF_c(X)\to\CF(Y)$ is invertible for all dimensions $n$.
\end{theorem}
\begin{proof}
Consider as the dual kernel $K'$ the indicator function of the set
\[
    \{(x,(\xi, t)) \colon x\cdot\xi\geq t\} .
\]
One observes the following.

Denote by $K_x$ the set of all $(\xi, t)$ such that $x$ lies in the halfspace $x\cdot\xi\leq t$. Likewise with the dual fiber $K_x'$ reversing the inequality. The intersection $K_x\cap K_x'$ is the set of $(\xi, t)$ with the property that for each $\xi\in\sphere^{n-1}$, there is a unique $t$ at which $x\cdot\xi=t$. Thus, $\mu = \chi(K_x\cap K_x') = \chi(\sphere^{n-1}) = 1-(-1)^n$.

For $x\neq x'$, the intersection $K_x\cap K_{x'}'$ is the set of all $(\xi, t)$ such that $x\cdot\xi\leq t$ and $x'\cdot\xi\geq t$. For fixed $\xi\in\sphere^{n-1}$, the set of compatible $t$ is empty if $(x-x')\cdot\xi<0$ and is a compact interval when $(x-x')\cdot\xi\geq 0$.  Thus, $K_x\cap K_{x'}'$ is a compact contractible set, and $\lambda = \chi(K_x\cap K_{x'}') = 1$.

As $\lambda\neq\mu$, the transform is invertible for all $n$.
\end{proof}


Note that we restrict our attention to only compactly-supported functions on $X$ in order to use the compactly-supported version of Euler characteristic, which is compatible with our definition of the Euler integral.

\begin{corollary}
The persistent homology transform of \cite{TMB} and the smoothed Euler characteristic transform of \cite{CMCMR} are invertible on constructible subsets of $\real^n$ for all $n$.
\end{corollary}

\section{Additional invertible transforms.}
The sublevelset Euler integral transform is but one of several invertible transforms on $X=\real^n$. As the Euler calculus appears underutilized, and as these transforms are so simple to define and invert, it seems appropriate to recall some known invertible topological integral transforms.
\begin{enumerate}
\item The original example of Schapira's inversion formula has $Y$ equal to the affine Grassmannian of hyperplanes in $X=\real^n$. Thus, recording all Euler characteristics of all flat codimension-1 slices is an invertible transform (with self-dual kernel).
\item The article \cite{BGL} gives several other examples of invertible transforms, including the following. Let $C$ be a compact convex definable subset of $X=\real^n=Y$ with kernel $K$ the indicator function on the set $\{x-y\in C\}$. Thus, $\Radon_K$ is a constructible ``blur'' with filter $C$. This is an invertible transform for all $n$.
\end{enumerate}

These examples are far from exhaustive. To close, we present a few novel invertible topological integral transforms.
\begin{enumerate}
\item Schapira's original example with the affine Grassmannian has a stereographic variant. Let $X=\Disc^n$ be a closed ball and $Y=\partial\Disc\times\real^{\geq 0}$. The (self-dual) kernel is given as the indicator function on the set $\{||x-y||=t\}$: one measures distance to a point on the boundary of $X$. The resulting transform is invertible for all $n$ with $\mu=\chi(\Sphere^{n-1})$ and $\lambda=\chi(\Sphere^{n-2})$.
\item The previous example can be modified to a sublevel/superlevel setting, analogous to the persistent Euler integral transform of this note. Keeping $X$ and $Y$ as before, one can set $K$ to be the indicator function on the set $\{||x-y||\leq t\}$ with the dual kernel $K'$ reversing the inequality. This transform is invertible for all $n$ with $\mu$ and $\lambda$ unchanged. These two examples suggest generalizations to other geometric domains with boundary.
\item Let $X=\real^n=Y$ with $\gamma$ a codimension-0 cone in $\real^n$ with vertex at the origin that does not contain a half-space. Let $K=K'$ be the indicator function over the set $\{(x, y)\colon x-y\in(\gamma\cup -\gamma)$. Then, for all $n>1$, this transform is invertible with $\mu = -1$ and $\lambda = 0$.
\end{enumerate}

In the same manner that the persistent Euler integral transform is discretized (and smoothed) to vectorize shape data \cite{TMB,CMCMR}, one can discretize any of the invertible Euler integral transforms defined above to use as a statistic for shapes (or more general constructible functions).


\end{document}